\documentclass{article}
\usepackage[utf8]{inputenc}
\usepackage{enumerate}
\usepackage{amssymb, bm}
\usepackage{amsthm}
\usepackage{amsmath}
\usepackage{tikz-cd} 
\usepackage{bbold}
 \usepackage[all]{xy}
\title{On Junyan Cao's vanishing theorem for pseudoeffective line bundles}
\author{Xiaojun WU}
\date{\today}
\newtheorem{mythm}{Theorem}
\newtheorem{mylem}{Lemma}
\newtheorem{myprop}{Proposition}
\newtheorem{mycor}{Corollary}

\newtheorem{myrem}{Remark}
\begin{document}
\def\cI{\mathcal{I}}
\def\Z{\mathbb{Z}}
\def\Q{\mathbb{Q}}  \def\C{\mathbb{C}}
 \def\R{\mathbb{R}}
 \def\N{\mathbb{N}}
 \def\H{\mathbb{H}}
  \def\P{\mathbb{P}}
 \def\rC{\mathrm{C}}
  \def\d{\partial}
 \def\dbar{{\overline{\partial}}}
\def\dzbar{{\overline{dz}}}
 \def\ii{\mathrm{i}}
  \def\d{\partial}
 \def\dbar{{\overline{\partial}}}
\def\dzbar{{\overline{dz}}}
\def \ddbar {\partial \overline{\partial}}
\def\cK{\mathcal{K}}
\def\cE{\mathcal{E}}  \def\cO{\mathcal{O}}
\def\P{\mathbb{P}}
\def\cI{\mathcal{I}}
\def \loc{\mathrm{loc}}
\def \log{\mathrm{log}}
\def \cC{\mathcal{C}}
\bibliographystyle{plain}
\def \deg{\mathrm{deg}}
\def \nd{\mathrm{nd}}
\def \liminf{\mathrm{liminf}}
\def \ker{\mathrm{Ker}}
\maketitle
\begin{abstract}
Junyan Cao has obtained a very general vanishing theorem, valid on any compact K\"ahler manifold, for the cohomology groups with values in a pseudoeffective line bundle twisted by the associated multiplier ideal sheaf. In this note, we give a version of this result in terms of the numerical dimension of the line bundle, instead of the numerical dimension of a given singular metric, which can be smaller.
\end{abstract}

\section{Introduction}
In \cite{Cao17}, Junyan Cao has proven the following very general Kawamata-Viehweg-Nadel type vanishing theorem.
\begin{mythm}
Let $(L,h)$ be  a  pseudo-effective  line  bundle  on  a  compact K\"ahler $n$-dimensional manifold $X$.  Then 
$$H^q(X,K_X \otimes L \otimes \cI(h)) = 0$$ for every $q \geq n-\nd(L,h) + 1$.
\end{mythm}
The numerical dimension $\nd(L,h)$ used in Cao's theorem is the numerical dimension of the closed positive $(1,1)$-current $i \Theta_{L,h}$ defined in his paper. Since we will not need this definition, we refer to his paper for further information.
We just recall the remark on page 22 of \cite{Cao17}.
In the example 1.7 of \cite{DPS94}, they consider the nef line bundle $\cO(1)$ over the projectivisation of a rank 2 vector bundle over the elliptic curve $C$ which is the only non-trivial extension of $\cO_C$. They prove that there exists a unique positive singular metric $h$ on $\cO(1)$.
For this metric, $\nd(\cO(1),h)=0$.
But the numerical dimension of $\cO(1)$ is equal to 1. 
We recall that for a nef line bundle $L$ the numerical dimension is defined as
$$\nd(L):= \max \{p ; c_1(L)^p \neq 0\}.$$
We also remark that Cao’s technique of proof actually yields the result for the upper semi-continuous regularization of multiplier ideal sheaf defined as
$$\cI_+(h) := \lim_{\varepsilon \to 0}\cI(h^{1+\varepsilon})$$
instead of $\cI(h)$, but we can apply Guan-Zhou’s Theorem  \cite{GZ14a}\cite{GZ13} \cite{GZ15} \cite{GZ14b} to see that the equality $\cI_+(h) =\cI(h)$ always holds. 
In particular, by the Noetherian property of ideal sheaves, we have
$$\cI_+(h)=\cI(h^{\lambda_0} )=\cI(h)$$
for some $\lambda_0 >1$.
This fact will also be used in our result.

In this note, we prove the following version of Junyan Cao's vanishing theorem, following closely the ideas of Junyan Cao \cite{Cao17} and the version that was a bit simplified in \cite{Dem14}.
\begin{mythm}
Let $L$ be  a  pseudo-effective  line  bundle  on  a  compact K\"ahler $n$-dimensional manifold $X$.  Then the morphism induced by inclusion $K_X \otimes L \otimes \cI(h_{\min}) \to K_X \otimes L$
$$H^q(X,K_X \otimes L \otimes \cI(h_{\min})) \to H^q(X,K_X \otimes L )$$ is 0 map for every $q \geq n-\nd(L) + 1$.
\end{mythm}
\begin{myrem}
{\rm
In the example 1.7 of \cite{DPS94}, since the rank 2 vector bundle is the only non-trivial extension of $\cO_C$, there exists a surjective morphism from this vector bundle to $\cO_C$ which induces a closed immersion $C$ into the ruled surface.
The only positive metric on $\cO(1)$ has curvature $[C]$ the current associated to $C$.
On the other hand, $\cO(1)=\cO(C)$. 
So we have
$H^2(X, K_X \otimes \cO(1))=H^0(X, \cO(-1))=H^0(X, \cO(-C))=0$
and $H^2(X, K_X \otimes \cO(1) \otimes \cI(h_{\min}))=H^2(X, K_X \otimes \cO(1) \otimes \cO(-C))=H^0(X, \cO_X)=\C$.
This shows that to get a numerical dimension version of theorem the best that we can hope for is that the morphism is 0 map instead of that $H^q(X, K_X \otimes L \otimes \cI(h_{\min}))=0$.
We notice that in general one would expect the vanishing result
$$H^q(X, K_X \otimes L)=0$$
for $q \geq n - \nd(L)+1$, whenever $L$ is a nef line bundle. Here the difficulty is to prove a general K\"ahler version, since the results follows easily from an inductive hyperplane section argument when $X$ is projective (cf. eg. Corollary (6.26) of \cite{Dem12}).
}
\end{myrem}
By remark 1 in \cite{Wu19}, we can even assume that $h_{\varepsilon}$ as stated in the definition of the numerical dimension is increasing to $h_{\min}$ as $\varepsilon \to 0$.
What we need here is that the weight functions $\varphi_{\varepsilon}$ has limit $\varphi_{\min}$ and is pointwise at least equal to $\varphi_{\min}$ with a universal upper bound on $X$.

Before giving the proof of the vanishing theorem, we give the general lines of the ideas and compare it with Cao's theorem.
The idea is using the $L^2$ resolution of the multiplier ideal sheaf and proving that every $\dbar$-closed $L^2(h_{\min})$ global section can be approximated by $\dbar$-exact $L^2(h_{\infty})$ global sections with $h_\infty$ some smooth reference metric on $L$.
To prove it, we solve the $\dbar$-equation using a Bochner technique with error term (as in \cite{DP02}), and we prove that the error term tends to 0.

For this propose, we need to estimate the curvature asymptotically by some special approximating hermitian metrics constructed by means of the Calabi-Yau theorem.
Cao tried to prove that the error term tends to 0 in the topology induced by the $L^2$-norm, with respect to the given singular metric. In this way, he tried to keep the multiplier ideal sheaf unchanged when approximating the singular metric, by means of suitable ``equisingular approximation''.
For our propose, we try to prove that the error term tends to 0 in the topology induced by $L^2$-norm with respect to some (hence any) smooth metric. 
It would be enough for us that the multiplier ideal sheaf of $h_{\min}$is included in the multiplier ideal sheaf of the approximating hermitian metric.
In some sense, Cao's theorem is more precise in studying the singularity of the metric which somehow explains why his approach works for any singular metric while our approach applies only for the image of the natural inclusion.  
\section{Proof of the main theorem}
We start the proof of the vanishing theorem by the following technical curvature and singularity estimate.
\begin{myprop}
Let $(L,h_{\min})$ be a pseudo-effective line bundle on a compact K\"ahler manifold $(X,\omega)$. Let us write $T_{\min}=\frac{i}{2\pi}\Theta_{L,h_{\min}}=\alpha+\frac{i}{2\pi}\d \dbar \varphi_{\min}$ where $\alpha$ is the curvature of some smooth metric $h_{\infty}$ on $L$ and $\varphi_{\min}$ is a quasi-psh potential. Let $p=\nd(L)$ be the numerical dimension of $L$. Then, for every $\gamma\in{}]0,1]$ and $\delta\in{}]0,1]$, there exists a quasi-psh potential $\Phi_{\gamma,\delta}$ on $X$ satisfying the following properties$\,:$
\item{\rm(a)} $\Phi_{\gamma,\delta}$ is smooth in the complement $X\setminus Z_{\delta}$ of an analytic set $Z_{\delta}\subset X$.
\vskip2pt
\item{\rm(b)} $\alpha+\delta\omega+\frac{i}{2 \pi} \d \dbar \Phi_{\gamma,\delta}\geq \frac{\delta}{2}(1-\gamma)\omega$ on $X$.
\vskip2pt
\item{\rm(c)} $(\alpha+\delta\omega+\frac{i}{2 \pi} \d \dbar\Phi_{\gamma,\delta})^n\geq a\,\gamma^n\delta^{n-p}\omega^n$ on $X\setminus Z_\delta$.
\vskip2pt
\item{\rm(d)} $\sup_X\Phi_{1,\delta}=0$, and for all $\gamma\in{}]0,1]$ there are estimates $\Phi_{\gamma,\delta}\le A$ and
$$\exp\big(-\Phi_{\gamma,\delta}\big)\le 
e^{-(1+b\delta)\varphi_{\min}}\exp\big(A-\gamma\Phi_{1,\delta}\big)$$
\item{\rm(e)} For $\gamma_0,\,\delta_0>0$ small,
$\gamma\in{}]0,\gamma_0]$, $\delta\in{}]0,\delta_0]$, we have
$$\cI_+(\varphi_{\min})=\cI(\varphi_{\min})\subset \cI(\Phi_{\gamma,\delta}).$$
Here $a,\,b,\,A,\,\gamma_0,\,\delta_0$ are suitable constants independent of $\gamma$, $\delta$.
\end{myprop}
\begin{proof}
Denote by $\psi_{\varepsilon}$ the (non-increasing) sequence of weight functions as stated in the definition of numerical dimension.
We have $\psi_{\varepsilon}\ge\varphi_{\min}$ for all $\varepsilon >0$, the $\psi_{\varepsilon}$ have analytic singularities and 
$$\alpha+\frac{i}{2 \pi}\d\dbar \psi_{\varepsilon}\ge-\varepsilon \omega.$$
Then for $\varepsilon\le\frac{\delta}{4}$, we have
\begin{align*}
   \alpha+\delta\omega+\frac{i}{2\pi}\d \dbar \big((1+b\delta)\psi_{\varepsilon}\big)  &\geq \alpha+\delta\omega-(1+b\delta)(\alpha+\varepsilon\omega) \\
   &\ge\delta\omega-(1+b\delta)\varepsilon\omega-b\delta\alpha\ge
{\textstyle\frac{\delta}{2}}\omega
\end{align*}
for $b \in ]0, \frac{1}{5}]$ small enough such that $\omega -b \alpha \ge 0$.

Let $\mu:\widehat X\to X$ be a log-resolution of $\psi_{\varepsilon}$, so that 
$$\mu^*\big(\alpha+\delta\omega+\frac{i}{2 \pi} \d \dbar ((1+b\delta)\psi_{\varepsilon})\big)=[D_{\varepsilon}]+\beta_{\varepsilon}$$
where $\beta_{\varepsilon}\ge\frac{\delta}{2}\mu^*\omega\ge 0$ is a smooth closed $(1,1)$-form on $\widehat X$ that is strictly positive in the complement $\widehat X\setminus E$ of the exceptional divisor, and $D_{\varepsilon}$ is an effective $\R$-divisor that includes all components $E_\ell$ of $E$. The map $\mu$ can be obtained by Hironaka \cite{Hir64} as a composition of a sequence of blow-ups with smooth centres, and we can even achieve that $D_{\varepsilon}$ and $E$ are normal crossing divisors.
For arbitrary small enough numbers $\eta_\ell>0$, $\beta_\varepsilon-\sum\eta_\ell[E_\ell]$ is a K\"ahler class on~$\widehat X$.
Hence we can find a quasi-psh potential $\widehat\theta_{\varepsilon}$ on $\widehat X$ such that $\widehat\beta_{\varepsilon}:=\beta_\varepsilon-\sum\eta_\ell[E_\ell]+\frac{i}{2 \pi} \d \dbar \widehat\theta_\varepsilon$ is a K\"ahler metric on~$\widehat X$.
By taking the $\eta_\ell$ small enough, we may assume that
$$\int_{\widehat X}(\widehat\beta_\varepsilon)^n\ge
\frac{1}{2}\int_{\widehat X}\beta_\varepsilon^n.$$
We will use Yau's theorem \cite{Yau78} to construct a form in the cohomology class of $\widehat{\beta_\varepsilon}$ with better volume estimate.
We have
\begin{align*}
\alpha+\delta\omega+\frac{i}{2 \pi}\d \dbar\big((1+b\delta)\psi_{\varepsilon}\big)
&\ge\alpha+\varepsilon \omega+\frac{i}{2 \pi}\d \dbar \psi_{\varepsilon}+(\delta-\varepsilon)\omega
-b\delta(\alpha+\varepsilon\omega)\\
&\ge(\alpha+\varepsilon \omega+\frac{i}{2 \pi} \d \dbar \psi_{\varepsilon})+\frac{\delta}{2}\omega.
\end{align*}
The assumption on the numerical dimension of $L$ implies the existence of a constant $c>0$ such that, with $Z=\mu(E)\subset X$, we have
\begin{align*}
\int_{\widehat X}\beta_\varepsilon^n&=
\int_{X\setminus Z}\big(\alpha+\delta\omega+\frac{i}{2 \pi} \d \dbar ((1+b\delta)\psi_{\varepsilon})\big)^n \\
&\ge{n\choose p}\Big(\frac{\delta}{2}\Big)^{n-p}\int_{X\setminus Z}
\big(\alpha+\varepsilon\omega+\frac{i}{2\pi}\psi_{\varepsilon}\big)^p\wedge\omega^{n-p}
\ge  c\delta^{n-p}\int_X\omega^n.
\end{align*}
Therefore, we may assume
$$\int_{\widehat X}(\widehat\beta_\varepsilon)^n\ge \frac{c}{2}\,\delta^{n-p}\int_X\omega^n.$$
We take $\widehat{f}$ a volume form on $\widehat{X}$ such that $\widehat f>\frac{c}{3}\delta^{n-p}\mu^*\omega^n$ everywhere on $\widehat X$ and such that $\int_{\widehat X} \widehat f=\int_{\widehat X}\widehat\beta_\varepsilon^n$.
By Yau's theorem \cite{Yau78}, there exists a quasi-psh potential $\widehat\tau_\varepsilon$ on $\widehat X$ such that 
\hbox{$\widehat\beta_\varepsilon+\frac{i}{2 \pi} \d \dbar \widehat\tau_\varepsilon$} is a K\"ahler metric on $\widehat X$ with the prescribed volume form $\widehat f>0$.

Now push our focus back to $X$. Set $\theta_\varepsilon=\mu_*\widehat\theta_\varepsilon$ and
$\tau_\varepsilon=\mu_*\widehat\tau_\varepsilon\in L^1_{loc}(X)$.
We define
$$\Phi_{\gamma,\delta}:=(1+b\delta)\psi_{\varepsilon}+\gamma(\theta_\varepsilon+\tau_\varepsilon).$$
By construction it is smooth in the complement $X\setminus Z_{\delta}$ i.e. property (a).
It satisfies
$$\mu^*\big(\alpha+\delta\omega+\frac{i}{2 \pi} \d \dbar ((1+b\delta)\psi_{\varepsilon}+
\gamma(\theta_\varepsilon+\tau_\varepsilon))\big)
=[D_\varepsilon]+(1-\gamma)\beta_\varepsilon+\gamma\Big(
\sum_\ell\eta_\ell[E_\ell]+\widehat\beta_\varepsilon+\frac{i}{2\pi} \d \dbar \widehat\tau_\varepsilon\Big)$$
$$\ge(1-\gamma)\beta_\varepsilon \ge\frac{\delta}{2}(1-\gamma)\,\mu^*\omega$$
since \hbox{$\widehat\beta_\varepsilon+\frac{i}{2 \pi} \d \dbar \widehat\tau_\varepsilon$} is a K\"ahler metric on $\widehat X$.
Thus the property (b) is satisfied.
Putting $Z_\delta=\mu(|D_\varepsilon|)\supset \mu(E)=Z$, we have on $X \setminus Z_\delta$
$$
\mu^* \big(\alpha+\delta\omega+\frac{i}{2 \pi}\d \dbar \Phi_{\gamma,\delta}\big)^n \ge \big(\beta_\varepsilon+\gamma\frac{i}{2 \pi} \d \dbar(\widehat \theta_\varepsilon+ \widehat\tau_\varepsilon) \big)^n
$$$$\ge\gamma^n\,(\widehat\beta_\varepsilon+\frac{i}{2 \pi} \d \dbar \widehat\tau_\varepsilon)^n\ge \frac{c}{3}\,\gamma^n\delta^{n-p}\mu^*\omega^n.
$$
Since $\mu: \widehat{X} \setminus D_\varepsilon \to X \setminus Z_\delta$ is a biholomorphism, the condition (c) is satisfied if we set $a=\frac{c}{3}$.

We adjust constants in $\widehat\theta_\varepsilon+\widehat\tau_\varepsilon$ so that $\sup_X\Phi_{1,\delta}=0$.
Since $\varphi_{\min}\le\psi_{\varepsilon}\le\psi_{\varepsilon_0}\le A_0:=\sup_X\psi_{\varepsilon_0}$ for $\varepsilon \leq \varepsilon_0$ and
$$
\Phi_{\gamma,\delta}=(1+b\delta)\psi_{\varepsilon}+\gamma
\big(\Phi_{1,\delta}-\psi_{\varepsilon}\big)\geq 
(1+b\delta)\varphi_{\min}+\gamma\Phi_{1,\delta}
-\gamma A_0$$
and we have $\Phi_{\gamma,\delta}\le(1-\gamma+b\delta)A_0$.
Thus the property (d) is satisfied if we set $A:= (1+b)A_0$.

We observe that $\Phi_{1,\delta}$ satisfies $\alpha+\omega+dd^c\Phi_{1,\delta}\ge 0$ and $\sup_X\Phi_{1,\delta}=0$, hence $\Phi_{1,\delta}$ belongs to a compact 
family  of quasi-psh functions. 
By theorem 2.50 a uniform version of Skoda’s integrability theorem in \cite{GZ17}, there exists a uniform small constant $c_0>0$ such that $\int_X
\exp(-c_0\Phi_{1,\delta})dV_\omega<+\infty$ for all $\delta\in{}]0,1]$.
If $f\in\cO_{X,x}$ is a germ of holomorphic function and $U$ a small
neighbourhood of~$x$, the H\"older inequality combined with estimate~(d)
implies
$$
\int_U|f|^2\exp(-\Phi_{\gamma,\delta})dV_\omega\le
e^A\Big(\int_U|f|^2e^{-p(1+b\delta)\varphi_{\min}}dV_\omega\Big)^{\frac{1}{p}}
\Big(\int_U|f|^2e^{-q\gamma\Phi_{1,\delta}}dV_\omega\Big)^{\frac{1}{q}}.
$$
Take $p\in{}]1,\lambda_0[$ (say $p=(1+\lambda_0)/2$),  and take
$$\gamma\le\gamma_0:=\frac{c_0}{q}=c_0\frac{\lambda_0-1}{\lambda_0+1}\quad
\hbox{and}\quad\hbox{$\delta \le\delta_0\in{}]0,1]$ so small that 
$p(1+b\delta_0)\le\lambda_0$.}$$
Then  $f\in\cI_+(\varphi_{\min})=\cI(\lambda_0\varphi_{\min})$ implies $f\in\cI(\Phi_{\gamma,\delta})$ which proves the condition (e).
\end{proof}
The rest of the proof follows from the proof of \cite{Cao17} (cf. also \cite{Dem14}, \cite{DP02}, \cite{Mou}).
We will just give an outline of the proof for completeness.

Let $\{f\}$ be a cohomology class in the group $H^q(X,K_X\otimes L\otimes\cI(h_{\min}))$, $q\ge n-\nd(L)+1$. 
The sheaf $\cO(K_X\otimes L)\otimes\cI(h_{\min})$ can be resolved by the complex $(K^{\bullet},\dbar)$ where $K^i$ is the sheaf of $(n,i)$-forms $u$ such that both $u$ and $\dbar u$ are locally $L^2$ with respect to the weight $\varphi_{\min}$. 
So $\{f\}$ can be represented by a $(n,q)$-form $f$ such that both $f$ and $\dbar f$ are $L^2$ with respect to the weight $\varphi_{\min}$,
i.e.\ $\int_X|f|^2\exp(-\varphi_{\min})dV_\omega<+\infty$ and
$\int_X|\dbar f|^2\exp(-\varphi_{\min})dV_\omega<+\infty$. 

We can also equip $L$ by the hermitian metric $h_\delta$ defined by the quasi-psh weight $\Phi_\delta=\Phi_{\gamma_0,\delta}$ obtained in Proposition 1, with $\delta\in{}]0,\delta_0]$.
Since $\Phi_\delta$ is smooth on $X\setminus Z_\delta$, the Bochner-Kodaira inequality shows that for 
every smooth $(n,q)$-form $u$ with values in $K_X\otimes L$ that is 
compactly supported on $X\setminus Z_\delta$, we have
$$\Vert\dbar u\Vert_\delta^2+\Vert\dbar^* u\Vert_\delta^2\ge
2\pi\int_X(\lambda_{1,\delta}+\ldots+\lambda_{q,\delta}-q\delta)|u|^2e^{-\Phi_\delta}dV_\omega,$$
where $\Vert u\Vert_\delta^2:=\int_X|u|^2_{\omega,h_\delta}dV_\omega=
\int_X|u|_{\omega, h_{\infty}}^2e^{-\Phi_\delta}dV_\omega$.
The condition (b) of Proposition 1 shows that
$$
0<\frac{\delta}{2}(1-\gamma_0) \leq \lambda_{1,\delta}(x)\le\ldots\le\lambda_{n,\delta}(x)
$$
where $\lambda_{i, \delta}$ are at each point $x\in X$, the eigenvalues of $\alpha+\delta\omega+\frac{i}{2 \pi} \d \dbar \Phi_\delta$ with respect to the base K\"ahler metric~$\omega$.
In other words, we have up to a multiple $2 \pi$
$$
\Vert\dbar u\Vert_\delta^2+\Vert\dbar^* u\Vert_\delta^2+\delta\Vert u\Vert_\delta^2\ge
\int_X(\lambda_{1,\delta}+\ldots+\lambda_{q,\delta})|u|_{\omega, h_{\infty}}^2e^{-\Phi_\delta}dV_\omega.
$$
By the proof of theorem 3.3 in \cite{DP02}, we have the following lemma:
\begin{mylem}
For every $L^2$ section of $\Lambda^{n,q}T^*_X
\otimes L$ such that $\Vert f\Vert_\delta<+\infty$ and $\dbar f=0$ in the sense of distributions, there exists a $L^2$ section $v=v_\delta$ of $\Lambda^{n,q-1}T^*_X\otimes L$ and a $L^2$ section $w=w_\delta$ of $\Lambda^{n,q}T^*_X\otimes L$ such that $f=\dbar v+w$ with
$$
\Vert v\Vert_\delta^2+\frac{1}{\delta}\Vert w\Vert_\delta^2\le
\int_X\frac{1}{\lambda_{1,\delta}+\ldots+\lambda_{q,\delta}}|f|^2e^{-\Phi_\delta}dV_\omega.
$$
\end{mylem}
By lemma 1 and condition (d) of proposition 1, the error term $w$ satisfies the $L^2$ bound,
$$\int_X|w|_{\omega, h_{\infty}}^2e^{-A}dV_\omega\le
\int_X|w|_{\omega, h_{\infty}}^2e^{-\Phi_\delta}dV_\omega\le\int_X\frac{\delta}{\lambda_{1,\delta}+\ldots+\lambda_{q,\delta}}|f|_{\omega, h_{\infty}}^2e^{-\Phi_\delta}dV_\omega.
$$
We will show that the right hand term tends to 0 as $\delta \to 0$.
To do it, we need to estimate the ratio function $\rho_{\delta}:=\frac{\delta}{\lambda_{1,\delta}+\ldots+\lambda_{q,\delta}}$.
The ratio function is first estimated in \cite{Mou}.

By estimates (b,c) in Proposition 1, we have $\lambda_{j,\delta}(x)\ge \frac{\delta}{2}(1-\gamma_0)$ and
$\lambda_{1,\delta}(x)\ldots \lambda_{n,\delta}(x)\ge a\gamma_0^n\delta^{n-p}$.
Therefore we already find $\rho_\delta(x)\le 2/q(1-\gamma_0)$.
On the other hand, we have
$$
\int_{X\setminus Z_\delta}\lambda_{n,\delta}(x)dV_\omega\le
\int_X(\alpha+\delta\omega+dd^c\Phi_\delta)\wedge\omega^{n-1}
=\int_X(\alpha+\delta\omega)\wedge\omega^{n-1}
\le{\rm Const},
$$
therefore the ``bad set'' $S_\varepsilon\subset X\setminus Z_\delta$ of points $x$ where $\lambda_{n,\delta}(x)>\delta^{-\varepsilon}$ has a volume with respect to $\omega$ Vol$(S_\varepsilon)\le C\delta^\varepsilon$ converging to~$0$ as $\delta\to 0$.
Outside of $S_\varepsilon$, 
$$
\lambda_{q,\delta}(x)^q\delta^{-\varepsilon(n-q)}\ge \lambda_{q,\delta}(x)^q\lambda_{n,\delta}(x)^{n-q}\ge a\gamma_0^n\delta^{n-p}.
$$
Thus we have $\rho_\delta(x)\le C\delta^{1-\frac{n-p+(n-q)\varepsilon}{q}}$. 
If we take $q\ge n-\nd(L)+1$ and $\varepsilon>0$ small enough, the exponent of $\delta$ in the final estimate is strictly positive.
Thus there exists a subsequence $(\rho_{\delta_\ell})$, $\delta_\ell\to 0$, that tends almost everywhere to $0$ on~$X$.

Estimate (e) in Proposition 1 implies the
H\"older inequality
$$
\int_X\rho_\delta|f|_{\omega, h_{\infty}}^2\exp(-\Phi_{\delta})dV_\omega\le
e^A\Big(\int_X\rho_\delta^p|f|_{\omega, h_{\infty}}^2e^{-p(1+b\delta)\varphi_{\min}}dV_\omega\Big)^{\!\frac{1}{p}}
\Big(\int_X|f|_{\omega, h_{\infty}}^2e^{-q\gamma_0\Phi_{1,\delta}}dV_\omega\Big)^{\!\frac{1}{q}}
$$
for suitable $p,q >1$ as in the proposition.
$|f|^2_{\omega, h_{\infty}} \leq C$ for some constant $C >0$ since $X$ is compact.
Taking $\delta \to 0$ yields that $w_{\delta} \to 0$ in $L^2( h_{\infty})$ by Lebesgue dominating theorem.

$H^q(X,K_X \otimes L)$ is a finite dimensional Hausdorff vector space whose topology is induced by the $L^2$ Hilbert space topology on the space of forms.
In particular the subspace of coboundaries is closed in the space of cocycles.  
Hence $f$ is a coboundary which completes the proof.

For any singular positive metric $h$ on $L$, by definition, $h$ is more singular that $h_{\min}$ which implies that $\cI(h) \subset \cI(h_{\min})$.
A direct corollary of theorem 1 is the following.
\begin{mycor}
Let $(L,h)$ be  a  pseudo-effective  line  bundle  on  a  compact K\"ahler $n$-dimensional manifold $X$.  Then the morphism induced by inclusion $K_X \otimes L \otimes \cI(h_{}) \to K_X \otimes L$
$$H^q(X,K_X \otimes L \otimes \cI(h_{})) \to H^q(X,K_X \otimes L )$$ is 0 map for every $q \geq n-\nd(L) + 1$.
\end{mycor}
\section{Appendix: movable intersection}
\paragraph{}
In this appendix, we show that the movable intersection of cohomology classes defined in \cite{BDPP} coincides with the positive product defined in \cite{BEGZ} which might be well-known for experts.
In particular, using movable intersection instead of positive intersection, we can give a third equivalent definition of numerical dimension of a psef cohomology class.

To distinguish the notations, we will denote by $\langle ,\rangle $ for the positive product and $\langle\! \langle, \rangle\! \rangle$ for the movable intersection.
In other words, it shows that the numerical definition of the psef class $\alpha$ can either defined to be the largest number such that $\langle \alpha^p \rangle \neq 0$ or such that $\langle\! \langle \alpha^p \rangle \!\rangle \neq 0$.

We start by recalling the definition of the movable intersection given in Theorem 3.5 of \cite{BDPP}.
Let $(X, \omega)$ be a compact K\"ahler manifold and $\alpha$ be a psef class on $X$.
To simplify the notations, we only define $\langle\! \langle \alpha^p \rangle \!\rangle$ where the general case is similar. 
First assume that $\alpha$ is big.
To know the value of the product pairing with any $(n-p, n-p)$-smooth form, it is enough to 
know its value with a countable dense family of smooth forms.
Since for any $(n-p,n-p)$-smooth form $u$, $u=C \omega^{n-p}-(C \omega^{n-p}-u)$.
For $C >0$ big enough, both $C \omega^{n-p}$ and $C \omega^{n-p}-u$ is strongly positive forms (since $X$ is compact).
Thus it is enough to consider only a countable dense family of strongly positive forms.

Fix a smooth closed $(n-p,n-p)$ strongly-positive form $u$ on $X$. We select K\"ahler currents $T\in\alpha$ with
analytic singularities, and a log-resolution
$\mu:\tilde{X} \to X$ such that 
$$
\mu^* T=[E]+\beta
$$
where $[E]$ is the current associated to a $\R$-divisor and $\beta$ a semi-positive form.
We consider the direct image current
$\mu_*(\beta \wedge\ldots\wedge\beta)$.
Given two closed positive $(1,1)$-currents $T_1, T_2 \in \alpha$, we write $T_j= \theta+i \d \dbar \varphi_i$ ($j=1,2$) for some smooth form $\theta \in \alpha$.
Define $T:= \theta+ i \d \dbar \max(\varphi_1, \varphi_2)$. 
We get a current with analytic singularities less singular than these two currents.
By this way, if we change the representative $T$ with another current $T'$,
we may always take a simultaneous log-resolution $\mu$ such that
$\mu^* T'=[E']+\beta'$, and we can always
assume that $E'\le E$.
By calculation, we find 
$$
\int_{\tilde{ X}}\beta' \wedge\ldots\wedge\beta' \wedge\mu^* u
\ge\int_{\tilde{ X}}\beta \ldots\wedge\beta\wedge\mu^* u.
$$ 
In fact, we have
$$
\int_{\tilde{ X}}\beta' \wedge \beta \wedge \ldots\wedge\beta \wedge\mu^* u=\int_{\tilde{ X}}(\beta+[E]-[E']) \wedge \beta \wedge \ldots\wedge\beta \wedge\mu^* u
$$$$
\ge\int_{\tilde{ X}}\beta \ldots\wedge\beta\wedge\mu^* u.
$$ 
A similar substitution applies to change all $\beta'$ by $\beta$.

It can be shown that the closed positive currents $\mu_*(\beta \wedge\ldots
\wedge\beta )$ are uniformly bounded in mass.
In fact, for any K\"ahler metric $\omega$ in $X$, there exists a constant $C > 0$ such that
$C\{\omega\}-\alpha$ is a K\"ahler class. In other words, there exists some K\"ahler form $\gamma$ on $X$  in the cohomology class $C\{\omega\}-\alpha$. By
pulling back with $\mu$, we find 
$$C\mu^* \omega - ([E] + \beta) \equiv \mu^* \gamma,$$ hence
$\beta \equiv C\mu^* \omega - ([E] +\mu^* \gamma)$
where $\equiv$ means in the same cohomology class.
By performing again a substitution in the integrals, we find
$$\int_{\tilde{X}} \beta^k \wedge \mu^* \omega^{n-k} \leq C^k \int_{\tilde{X}} \mu^* \omega^n=C^k \int_{X} \omega^n.$$

For each of the integrals associated with a countable dense family 
of forms $u$,  the supremum of $\int_{\tilde{ X}}\beta \ldots\wedge\beta\wedge\mu^* u$ is achieved by a sequence of currents 
$(\mu_m)_*(\beta_{m}\wedge\ldots\wedge\beta_{m})$ obtained
as direct images by a suitable sequence of modifications $\mu_m:
\tilde{X}_m\to X$ and suitable $\beta_m$'s.
By extracting a subsequence, we can achieve that this sequence is weakly 
convergent and we set
$$
\langle\!\langle \alpha^p \rangle \!\rangle :=
\mathop{\lim\uparrow}\limits_{m\to+\infty}
\{(\mu_m)_*(\beta_{m}\wedge\ldots\wedge\beta_{m})\}
$$
In the general case when $\alpha$ is only psef, we define $$
\langle\! \langle \alpha^p \rangle \!\rangle :=
\mathop{\lim\downarrow}\limits_{\delta\downarrow 0}
\langle\! \langle (\alpha+\delta\{\omega\})^p \rangle \!\rangle.
$$

We now prove that these two products coincide for psef classes. Since in the two cases, the products are the limit of the products of big classes in $H^{p,p}(X, \R)$, without loss of generality, we can assume $\alpha$ to be big.
We state it in the following lemma.
\begin{mylem}
For $\alpha$ a big class, for any $1 \le p \le n$, we have
$$\langle \alpha^p \rangle=\langle\! \langle \alpha^p \rangle \! \rangle.$$
\end{mylem}
\begin{proof}
It is enough to prove by duality that for any smooth closed $(n-p,n-p)$-strongly positive form $u$ on $X$, we have
$$\int_X \langle \alpha^p \rangle \wedge u=\int_X \langle\! \langle \alpha^p \rangle \! \rangle \wedge u.$$
Denote by $A$ the non-K\"ahler locus of $\alpha$ which  is the pole of some K\"ahler current $T$ in $\alpha$ with analytic singularities.
Denote by $T_{\min} \in \alpha$, the current with minimal singularities in $\alpha$.
By definition, it is less singular than the K\"ahler current $T$.
In particular, the potential of $T_{\min}$ is locally bounded outside $A$.
Let $T_{\varepsilon}$ be a regularisation of $T_{\min}$ such that $T_{\varepsilon} \ge -\varepsilon \omega$.
Their potentials are locally bounded outside $A$ as $T_{\min}$'s is.
By weak continuity of the Bedford-Taylor Monge-Amp\`ere operator along decreasing sequences we have on $X \setminus A$ for any $\delta >0$ and $\varepsilon \le \delta$
$$(T_{\min}+\delta \omega)^p \to (T_{\varepsilon}+\delta \omega)^p$$
as current.
By Fatou lemma we have
$$\int_{X \setminus A} (T_{\min}+\delta \omega)^p \wedge u \leq \liminf_{\varepsilon \to 0} \int_{X \setminus A}(T_{\varepsilon}+\delta \omega)^p \wedge u.$$
Since the non-pluripolar product of currents has no mass along any analytic set, the left hand term has limit equal to
$\int_X \langle \alpha^p \rangle \wedge u$.
We remark that both $\langle \alpha^p \rangle$ and $\langle \! \langle \alpha^p \rangle \! \rangle$ depend continuously on $\alpha$ in the big cone. 
Since $T_\varepsilon$ has analytic singularities, there exists a modification $\mu:\tilde{X} \to X$ such that 
$$
\mu^* T_\varepsilon=[E]+\beta
$$
with $E$ associated to a $\R$-divisor and $\mu$ a biholomorphism on $X \setminus A$.
So we have
$$\int_{X \setminus A}(T_{\varepsilon}+\delta\omega)^p \wedge u = \int_{X \setminus A} \mu_* (\beta+\delta \mu^* \omega)^p \wedge u \leq \int_X \langle\! \langle (\alpha+\delta \omega)^p \rangle \! \rangle \wedge u.
$$
We remark that $T_\varepsilon+\delta \omega$ is a K\"ahler current with analytic singularities for $\varepsilon <\delta$. 
When $\delta \to 0$, we have
$$\liminf_{\varepsilon \to 0} \int_{X \setminus A}(T_{\varepsilon}+\delta \omega)^p \wedge u \le \int_X \langle\! \langle \alpha^p \rangle \! \rangle \wedge u.$$
For the other direction, for any $T \in \alpha$ a K\"ahler current with analytic singularities, there exists a modification $\mu:\tilde{X} \to X$ such that 
$$
\mu^* T=[E]+\beta
$$
as above.
By the definition of non K\"ahler locus, $T$ is locally bounded outside $A$.
The modification can be achieved by a composition of blow-ups with smooth centres in $A$, so $\mu$ is a biholomorphism on $X \setminus A$.
So we have
$$\int_{X \setminus A}T^p \wedge u = \int_{X \setminus A}\mu_* (\beta^p) \wedge u =\int_{\tilde{X}} \beta^p \wedge \mu^*u\leq \int_{X \setminus A} \langle T_{\min}^p \rangle \wedge u= \int_X \langle T_{\min}^p \rangle \wedge u.
$$
The inequality use proposition 1.16 in \cite{BEGZ} cited above and the fact that $T_{\min}$ is less singular than $T$.
By taking supremum among all K\"ahler currents with analytic singularities in the cohomology class $\alpha$, we have
$$\int_X \langle \alpha^p \rangle \wedge u \geq \int_X \langle\! \langle \alpha^p \rangle \! \rangle \wedge u.$$
\end{proof}
\textbf{Acknowledgement} I thank Jean-Pierre Demailly, my PhD supervisor, for his guidance, patience and generosity. 
I would also like to express my gratitude to colleagues of Institut Fourier for all the interesting discussions we had. This work is supported by the PhD program AMX of \'Ecole Polytechnique and Ministère de l'Enseignement Supérieur et de la Recherche et de l’Innovation, and the European Research Council grant ALKAGE number 670846 managed by J.-P. Demailly.

\end{document}